\documentclass{amsart}

\usepackage{lineno,hyperref}
\usepackage{amsmath,amssymb}

\newtheorem{thm}{Theorem}[section]
\newtheorem{defn}[thm]{Definition}

\newtheorem{cor}[thm]{Corollary}
\newtheorem{lem}[thm]{Lemma}
\newtheorem{remark}[thm]{Remark}

\newcommand{\ds}{\displaystyle}

\begin{document}

\title{Atomic subspaces for operators}

\author[Bhandari]{A. Bhandari}

\address{Dept. of Mathematics\\ NIT Meghalaya\\ Shillong 793003\\ India}

\email{animesh@nitm.ac.in}

\author[Mukherjee]{S. Mukherjee}
\address{Dept. of Mathematics\\ NIT Meghalaya\\ Shillong 793003\\ India}
\email{saikat.mukherjee@nitm.ac.in}
\thanks{Second author is supported by NIT Meghalaya Start-up Grant Project}


\begin{abstract}
This paper introduces the concept of atomic subspaces with respect to a bounded linear operator. Atomic subspaces generalize fusion frames and this generalization leads to the notion of $K$-fusion frames. Characterizations of $K$-fusion frames are discussed. Various properties of $K$-fusion frames, for example, direct sum, intersection, are studied.
\end{abstract}

\maketitle

\section{Introduction}\label{intro}
Notion of Hilbert space frames was first introduced by Duffin and Schaeffer \cite{Du52} in 1952 to reconstruct signals. Much later in the year 1986, the fundamental concept of frames and their significance in signal processing, image processing and data processing were presented by Daubechies, Grossman and Meyer  \cite{Da86}. Frame theory plays an important role in various fields and have been widely applied in signal processing \cite{Fe99}, sampling theory \cite{El03}, coding and communications (\cite{St03}, \cite{Ho04}) and so on.

It is a well-known fact that every element in a separable Hilbert space $\mathcal{H}$ can be explicitly represented as a linear combination of an orthonormal basis in $\mathcal H$ with the help of Fourier coefficients. But if one of the basis elements, for some reason, is removed, the explicit representation may not hold. Primarily due to this reason an overcomplete system was introduced which satisfies the explicit representation  but more flexible when $f\in \mathcal{H}$ is to be reconstructed. Such an overcomplete system is called a ``Frame''.

L. G\v{a}vru\c{t}a in \cite{Ga12} was first to introduce the notion of $K$-frames to study the nature of {\it atomic systems} for a separable Hilbert space $\mathcal H$ with respect to a bounded linear operator $K$ on $\mathcal H$. In \cite{Ga15}, G\v{a}vru\c{t}a further studied atomic systems for operators in reproducing kernel Hilbert spaces, especially Bergman and Fock spaces. It is well-known fact that $K$-frames are more general than the classical frames and due to higher generality of $K$-frames, many properties of frames may not hold for $K$-frames.

In the 21st century scientists introduced fusion frames to handle massive amount of data to obtain mathematical framework to model and analyze such problems, which are otherwise almost impossible to handle. Moreover fusion frames are also significantly important mathematical gadget for theory oriented mathematical problems in frame theory. The notion of fusion frames (or frames of subspaces) was first introduced by Casazza et. al. (see \cite{Ca04}, \cite{Ca08}). There are so many applications  of fusion frames like coding theory, compressed sensing, data processing and so on. A fusion frame is a frame-like collection of closed subspaces in a Hilbert space. In frame theory, amplitudes of projection vectors onto frame elements are used to represent signals whereas in the fusion frame theory, signals are represented by its projection vectors onto fusion frame subspaces. Also more specifically we may acquire that fusion frames are the generalization of conventional classical frames and special cases of $g$-frames in the field of frame theory.

This paper presents notion of {\it atomic subspaces} with respect to a bounded linear operator on a separable Hilbert space which leads to the concept of $K$-fusion frames, a generalization of fusion frames. This also generalize some results of \cite{Ga12}.

The paper is organized as follows. In Section \ref{Sec-Preli}, we recall basic definitions and results related to frames, $K$-frames and fusion frames. Atomic subspaces and $K$-fusion frames are introduced and discussed in Section \ref{Sec-atom}. Finally in Section \ref{Sec-result} we characterize $K$-fusion frames and establish various properties of the same.

Throughout the paper, $\mathcal{H}$ is a separable Hilbert space. We denote by $\mathcal{L}(\mathcal{H}_1, \mathcal{H}_2)$ the space of all bounded linear operators from  $\mathcal{H}_1$ into $\mathcal{H}_2$. For $T \in \mathcal{L}\mathcal{(H)}$, we denote $D(T), N(T)$ and $R(T)$ for domain, null space and range of $T$, respectively. We consider the index set $I$ to be finite or countable.
\section{Preliminaries}\label{Sec-Preli}
In this section we recall basic definitions and results needed in this paper. We refer to the book by Ole Christensen \cite{Ch03} for an introduction to frame theory.

\subsection{Frame} A collection  $\{ f_i \}_{i\in I}$ in $\mathcal{H}$ is called a \emph{frame} if there exist constants $A,B >0$ such that \begin{equation}\label{Eq:Frame} A\|f\|^2~ \leq ~\sum_{i\in I} |\langle f,f_i\rangle|^2 ~\leq ~B\|f\|^2,\end{equation}for all $f \in \mathcal{H}$. The numbers $A, B$ are called \emph{frame bounds}. The supremum over all $A$'s and infimum over all $B$'s satisfying above inequality are called the \emph{optimal frame bounds}. The frame is called a \emph{tight frame} if $A=B$ and if $A=B=1$ it is called a \emph{Parseval frame}. The frame is called \emph{exact} if it ceases to be a frame whenever any single element is removed from the collection. If a collection satisfies only the right inequality in (\ref{Eq:Frame}), it is called a {\it Bessel sequence}.

Given a frame $\{f_i\}_{i\in I}$ of $\mathcal{H}$. The \emph{pre-frame operator} or \emph{synthesis operator} is a bounded linear operator $T: l^2(I)\rightarrow \mathcal{H}$ and is defined by $T\{c_i\} = \ds \sum_{i\in I} c_i f_i$. The adjoint of $T$, $T^*: \mathcal{H} \rightarrow l^2(I)$, given by $T^*f = \{\langle f, f_i\rangle\}$, is called the \emph{analysis operator}. The \emph{frame operator}, $S$, is obtained by composing $T$ with $T^*$ , $S=TT^*$. That is, $S: \mathcal{H}\rightarrow \mathcal{H}$  such that $$Sf=TT^*f = \sum_{i\in I} \langle f, f_i\rangle f_i.$$ The frame operator is bounded, positive, self adjoint and invertible.

\underline{\bf{Reconstruction formula}}: Every element in $\mathcal{H}$ can be represented using frame elements as follows:
\begin{equation}\label{Eq:Frame-recons}
f = \sum_{i\in I} \langle f, S^{-1}f_i\rangle f_i = \sum_{i\in I} \langle f, f_i\rangle S^{-1}f_i
\end{equation}
Since the frame elements are not necessarily linearly independent, this representation is not unique, in general.

\subsection{$K$-Frame} Let $K \in \mathcal L(\mathcal H)$, then a sequence $\lbrace f_i \rbrace_{i\in I}$ in $\mathcal H$ is called a {\it $K$-frame} for $\mathcal H$ if there exist positive constants $A, B$ such that
\begin{equation}\label{Eq:K-fram}
 A\|K^* f\|^2 \leq \sum_{i\in I} |\langle f,f_i \rangle |^2 \leq B\|f\|^2, \end{equation}
for all $f \in \mathcal H$ and the above sequence is said to be a {\it tight $K$-frame} if
\begin{equation}\label{Eq:K-tight-fram}
 A\|K^* f\|^2 = \sum_{i\in I} | \langle f,f_i \rangle |^2, \end{equation}
for all $f\in \mathcal H$.
\subsection{Fusion Frame} Given a Hilbert space $\mathcal H$, consider a collection of closed subspaces $\lbrace \mathcal W_i \rbrace_{i \in I}$  of $\mathcal H$ and a collection of positive weights $\lbrace v_i \rbrace_{i \in I}$. A family of weighted closed subspaces $(\mathcal W, v) = \lbrace (\mathcal W_i, v_i) : i\in I\rbrace$ is called a fusion frame for $\mathcal H$, if there exist constants $0<A \leq B <\infty$ satisfying
\begin{equation}\label{Eq:Fus-frame}
A\|f\|^2 \leq \sum_{i \in I} v_i ^2 \|P_{\mathcal W_i}f\|^2 \leq B\|f\|^2, \end{equation}
where $P_{\mathcal W_i}$ is the orthogonal projection from $\mathcal H$ onto $\mathcal W_i$. The constants $A$ and $B$ are called {\it fusion frame bounds}. If $A=B$ then the fusion frame is called a {\it tight fusion frame}, if $A=B=1$ then it is called a {\it Parseval fusion frame} and the fusion frame is called {\it orthonormal} if $\mathcal H = \bigoplus \mathcal W_i$. If $v_i=v,~\forall~ i\in I$, it is called {\it $v$-uniform fusion frame}. A collection of closed subspaces, satisfying only the right inequality in \ref{Eq:Fus-frame}, is called a {\it fusion Bessel sequence}.

For a family of closed subspaces, $\lbrace \mathcal W_i \rbrace_{i \in I}$, of $\mathcal H$, the corresponding $l^2$ space is defined by $\left(\sum_{i \in I} \bigoplus \mathcal W_i\right)_{l^2} = \lbrace \lbrace f_i \rbrace_{i \in I} : f_i \in \mathcal W_i, \sum_{i \in I} \|f_i\|^2 < \infty \rbrace$ with inner product is given by $\langle \lbrace f_i \rbrace, \lbrace g_i \rbrace \rangle = \sum_{i \in I} \langle f_i , g_i \rangle_{\mathcal{H}} $.

Let $\lbrace (\mathcal W_i, v_i) \rbrace_{i \in I}$ be a fusion frame. Then the synthesis operator $T_{\mathcal W}: (\sum_{i \in I} \bigoplus \mathcal W_i)_{l^2} \rightarrow \mathcal H$ is defined as $T_{\mathcal W} (f)=\sum_{i \in I} v_i f_i$ for all $f=\lbrace f_i \rbrace_{i \in I} \in (\sum_{i \in I} \oplus \mathcal W_i)_{l^2}$ and the analysis operator $T^*_{\mathcal W}: \mathcal H \rightarrow  (\sum_{i \in I} \oplus \mathcal W_i)_{l^2} $ is defined as  $T^*_{\mathcal W} (f)=\lbrace v_i P_{\mathcal W_i} (f) \rbrace_{i \in I} $. It is well-known that (see \cite{Ca04}) the synthesis operator $ T_{\mathcal W}$ of a fusion frame is bounded, linear and onto, whereas the corresponding analysis operator $ T_{\mathcal W} ^*$  is (possibly into) an isomorphism. Corresponding fusion frame operator is defined as $S_\mathcal W (f) = T_{\mathcal W}T_{\mathcal W}^*(f)=\sum_{i \in I} v_i ^2 P_{\mathcal W_i}(f)$.  $S_\mathcal{W}$ is bounded, positive, self adjoint and invertible.

{\bf Reconstruction formula:} Any signal $f \in \mathcal H$ can be expressed by its fusion frame measurements $\lbrace v_i P_{\mathcal W_i} f \rbrace_{i \in I}$ as
\begin{equation}\label{Eq:Fusion-frame-recons} f=\sum_{i \in I} v_i S_\mathcal W ^{-1} (v_i P_{\mathcal W_i} f) .\end{equation}

{\bf Orthonormal basis in $(\sum_{i \in I} \bigoplus \mathcal W_i)_{l^2}$:} Consider a family of closed subspaces $\lbrace \mathcal W_i \rbrace_{i \in I}$  of $\mathcal H$. Let  $\mathcal U = \lbrace u_1, u_2,....\rbrace$ be an orthonormal basis  for $\mathcal H$ and consider a family of sets $\lbrace J_i \rbrace_{i \in I}$ such that $J_i = \mathcal W_i \cap \mathcal U$ and denote the cardinality of $J_k$ by $|J_k|$. For each $k\in I$, define $I$-tuples, $e_k ^{l_k} = (\delta_i ^k J_k (l_k))_{i \in I}$, $l_k = 1,2,...,|J_k|$, where $J_k (l_k)$ is the $l_ k$-th element of $J_k$. It is easy to verify that the collection $\lbrace e_i ^{l_i} \rbrace_{i \in I}$ is countable and forms an orthonormal basis for $(\sum_{i \in I} \bigoplus \mathcal W_i)_{l^2}$.

We recall Douglas' factorization theorem (see \cite{Do66}) which is required to present few results.

\begin{thm}(Douglas' factorization theorem)\label{Thm-Douglas} Let $\mathcal{H}_1, \mathcal{H}_2,$ and $\mathcal H$ be Hilbert spaces and $S\in \mathcal L(\mathcal H_1, \mathcal H)$, $T\in \mathcal L(\mathcal H_2, \mathcal H)$. Then the following are equivalent:
\begin{enumerate}
\item $R(S) \subseteq R(T)$.
\item $SS^* \leq \alpha TT^*$ for some $\alpha >0$.
\item $S = TL$ for some $L \in \mathcal L(\mathcal H_1, \mathcal H_2)$.
\end{enumerate}
\end{thm}

\section{Atomic subspaces}\label{Sec-atom}

We define atomic subspace of $\mathcal H$ with respect to a bounded linear operator.
\begin{defn}\label{Defn:atom}
Let $K \in \mathcal L(\mathcal H)$ and consider a family of closed subspaces $\lbrace \mathcal W_i \rbrace_{i \in I}$ and a family of positive weights $\lbrace v_i \rbrace_{i \in I}$. Then $\lbrace (\mathcal W_i ,v_i)\rbrace_{i \in I}$ is said to be an {\it atomic subspace} of $\mathcal H$ with respect to $K$ if the following conditions hold:
\begin{enumerate}
\item[(a)] $\sum_{i \in I} v_i f_i$ is convergent for all $\lbrace f_i \rbrace_{i \in I} \in  (\sum_{i \in I} \oplus \mathcal W_i)_{l^2} $.

\item[(b)] For every $f \in \mathcal H$, there exist $\lbrace f_i \rbrace_{i \in I} \in  (\sum_{i \in I} \oplus \mathcal W_i)_{l^2} $ such that $K f=\sum_{i \in I} v_i f_i$ and $\| \lbrace f_i \rbrace\|_{(\sum_{i \in I} \oplus \mathcal W_i)_{l^2}} \leq C \|f\|_{\mathcal H}$ for some $C>0$.
\end{enumerate}
\end{defn}

\begin{remark}
Condition (a) in Definition \ref{Defn:atom} is equivalent to say that $\lbrace (\mathcal W_i ,v_i)\rbrace_{i \in I}$ is a fusion Bessel sequence.
\end{remark}
In the following we present the existence theorem of atomic subspaces.

\begin{thm}
A separable Hilbert space has an atomic subspace with respect to every bounded linear operator.
\end{thm}

\begin{proof}
Let $K \in \mathcal L(\mathcal H)$ and consider $\lbrace e_n \rbrace_{n \in \mathbb N}$ as an orthonormal basis for $\mathcal H$. Define $\mathcal U_n =span \lbrace e_n \rbrace$ and $\mathcal W_n = K(\mathcal U_n)$ for $n\in \mathbb{N}$. Then $\lbrace \mathcal U_n \rbrace$, $\lbrace \mathcal W_n \rbrace$ form sequences of closed subspaces of $\mathcal H$. Also define $v_n = \|K e_n \|$, $n\in \mathbb{N}$. We claim that $\lbrace (\mathcal W_n, v_n) \rbrace_n$ forms an atomic subspace of $\mathcal H$ with respect to $K$.

To prove this first note that for every $f\in \mathcal H$, $P_{\mathcal W_n} f = \frac{\langle f, Ke_n\rangle}{v_n^2} K e_n$. Hence we have  $$\sum_{n \in \mathbb N} v_n ^2 \|P_{\mathcal W_n} f\|^2 = \sum_{n \in \mathbb N} | \langle f, K e_n \rangle |^2 = \|K^*f\|^2 \leq \|K\|^2 \|f\|^2.$$ This shows that $\lbrace (\mathcal W_n, v_n) \rbrace_n$ is a  fusion Bessel sequence.

Again for all $f \in \mathcal H$, $f= \sum_{n \in \mathbb N} \langle f,  e_n \rangle e_n $ and therefore $K f = \sum_{n \in \mathbb N} \langle f,  e_n \rangle K e_n = \sum_{n \in \mathbb N} v_nf_n$, where  $f_n=\frac {\langle f,  e_n \rangle}  {\|K e_n \|}  K e_n\in\mathcal W_n$. Thus we have

$$\| \lbrace f_n \rbrace \|_{(\sum_{n \in \mathbb N} \oplus \mathcal W_n)_{l^2}}^2 = \sum_{n \in \mathbb N} \|f_n \|^2 = \sum_{n \in \mathbb N} | \langle f,  e_n \rangle |^2 = \|f\|^2.$$
\end{proof}

The notion of atomic subspaces has revived to produce generalization of family of local atoms or atomic systems for a bounded, linear operator. The following theorem provides a characterization of atomic subspaces.

\begin{thm}\label{Thm-Atom-Kfusion}
Let $\mathcal H$ be a Hilbert space. Assume that $\lbrace \mathcal W_i \rbrace_{i \in I}$ be a family of closed subspaces of $\mathcal H$ and $\lbrace  v_i \rbrace_{i \in I}$ be a family of positive weights. Then the following statements are equivalent :
\begin{enumerate}
\item $\lbrace (\mathcal W_i ,v_i)\rbrace_{i \in I}$  is an atomic subspace of $\mathcal H$ with respect to $K$.\\
\item \label{Atom-Kfusion}There exist $A, B >0$ such that for all $f \in \mathcal H$,
$$A\|K^*f\|^2 \leq \sum_{i \in I} v_i ^2 \|P_{\mathcal W_i} f\|^2 \leq B\|f\|^2.$$
\end{enumerate}
\end{thm}

\begin{proof} Suppose $\lbrace (\mathcal W_i ,v_i)\rbrace_{i \in I}$  is an atomic subspace of $\mathcal H$ with respect to $K$. It is sufficient to show that there exists a constant $A >0$ such that $\sum_{i \in I} v_i ^2 \|P_{\mathcal W_i} f\|^2\geq A\|K^*f\|^2$ for all $f \in \mathcal H$. But since $\|T_\mathcal W^*f\|^2=\sum_{i \in I} v_i ^2 \|P_{\mathcal W_i} f\|^2$, where $T_\mathcal W$ is the corresponding synthesis operator, this is equivalent to show that $T_\mathcal WT_\mathcal W^*\geq A KK^*$. Now since $T_\mathcal W$ is bounded, linear, onto \cite{Ca04}, $R(T_\mathcal W) \supseteq R(K)$. Therefore by using Theorem \ref{Thm-Douglas} we get the desired result.

Conversely, suppose that the inequality in \ref{Atom-Kfusion} is true. Then the right inequality asserts that $\lbrace (\mathcal W_i ,v_i)\rbrace_{i \in I}$ is a fusion Bessel sequence. Now the left inequality gives $A K K^* \leq T_\mathcal W T_\mathcal W ^*$.  Then using Theorem \ref{Thm-Douglas}, there exists a bounded linear operator $L \in \mathcal L(\mathcal H, (\sum_{i \in I} \oplus \mathcal W_i)_{l^2})$ such that $K=T_\mathcal W L$. For every $f\in \mathcal H$, define $Lf=\lbrace f_i \rbrace_{i \in I}$. Therefore $Kf=T_\mathcal W \lbrace f_i \rbrace_{i \in I}=\sum_{i \in I} v_i f_i$ and  $\| \lbrace f_i \rbrace \|_{(\sum_{i \in I} \oplus \mathcal W_i)_{l^2}} = \|Lf \|_{(\sum_{i \in I} \oplus \mathcal W_i)_{l^2}} \leq \|L\| \|f\|$ for all $f \in \mathcal H$.
This completes the proof.
\end{proof}

\begin{cor} Let  $\lbrace (\mathcal W_i ,v_i)\rbrace_{i \in I}$ be a fusion Bessel sequence in $\mathcal H$. Then  $\lbrace (\mathcal W_i ,v_i)\rbrace_{i \in I}$ is an atomic subspace of $\mathcal H$ with respect to the corresponding fusion frame operator  $S_\mathcal W f=\sum_{i \in I} v_i ^2 P_{\mathcal W_i} (f)$, for all $f \in \mathcal H$.
\end{cor}

\begin{proof} Given that $\lbrace (\mathcal W_i ,v_i)\rbrace_{i \in I}$ be a fusion Bessel sequence in $\mathcal H$. Then $\sum_{i \in I} v_i ^2 \|P_{\mathcal W_i} f\|^2 \leq B\|f\|^2$ for all $f \in \mathcal H$ and for some $B>0$. Now since $R(T_\mathcal W) = \mathcal H = R(S_\mathcal W)$, by using Douglas theorem (\ref{Thm-Douglas}), we have $A S_\mathcal W S_\mathcal W ^* \leq T_\mathcal W T_\mathcal W ^*$, for some $A>0$. Hence $A \|S_\mathcal W ^* f\|^2 \leq \sum_{i \in I} v_i ^2 \|P_{\mathcal W_i} f\|^2 $ and hence the result follows from Theorem \ref{Thm-Atom-Kfusion}.
\end{proof}

Theorem \ref{Thm-Atom-Kfusion} provides a generalization of fusion frames.
\begin{defn}
Given $\mathcal H$, a collection of closed subspaces $\lbrace \mathcal W_i \rbrace_{i \in I}$  of $\mathcal H$ with a collection of positive weights $\lbrace v_i \rbrace_{i \in I}$, $\lbrace (\mathcal W_i, v_i) : i\in I\rbrace$, is said to be a $K$-fusion frame for $\mathcal H$ with respect to $K\in \mathcal L(\mathcal H)$ if there exist positive constants $A, B$ such that
\begin{equation}\label{Eq:K-Fus-frame}
A\|K^*f\|^2 \leq \sum_{i \in I} v_i ^2 \|P_{\mathcal W_i}f\|^2 \leq B\|f\|^2. \end{equation}
\end{defn}
In this context, we acknowledge that recently Liu and Li \cite{Li18} introduced the concept of $K$-fusion frames. They studied $K$-fusion frames with unitary systems' structure and introduced the concept of $K$-fusion frame generators. The present work has been done almost simultaneously with the work of Liu and Li.

Here we recall the definition of  Moore-Penrose pseudo inverse of an operator.

\begin{defn}\cite{Jo13} Let $\mathcal H$ be a Hilbert space and suppose that $T \in \mathcal L(\mathcal H)$ has closed range. Then there exists an operator $T^\dag \in \mathcal L(\mathcal H)$ for which
$$N(T^\dag)=N(T^*),~~~R(T^\dag)=R(T^*),~~~TT^\dag = P_{R(T)},~~~T^\dag T = P_{R(T^*)}.$$
$T^\dag$ is called Moore-Penrose pseudo inverse of $T$ and is uniquely determined by the above mentioned  properties. If $T$ is invertible, then $T^{-1}=T^\dag$.
\end{defn}

The following theorem provides a relation between fusion frames and $K$-fusion frames.

\begin{thm} Let $K \in \mathcal L(\mathcal H)$. Then:
 \begin{itemize}
 \item[(a)] Every fusion frame is a $K$-fusion frame.
 \item[(b)] If $R(K)$ is closed, every $K$-fusion frame is a fusion frame for $R(K)$.
 \end{itemize}
\end{thm}

\begin{proof}

\begin{itemize}
\item[(a)] Let $\lbrace (\mathcal W_i, v_i)\rbrace_{i \in I}$ be a fusion frame for $\mathcal H$ with frame bounds $A, B$. Then for all $f\in \mathcal H$,
$$\frac {A} {\|K\|^2} \|K^*f\|^2 \leq  A \|f\|^2 \leq \sum_{i\in I}v_i^2\|P_{\mathcal W_i}f\|^2\leq B\|f\|^2.$$

\item[(b)] Let $\lbrace (\mathcal W_i, v_i)\rbrace_{i \in I}$ be a $K$-fusion frame for $\mathcal H$ with frame bounds $A, B$. Then for all $f\in R(K)$, $$\frac {A} {\|K^{* \dag}\|^2 } \|f\|^2 \leq  A  \|K^*f\|^2\leq\sum_{i\in I}v_i^2\|P_{\mathcal W_i}f\|^2\leq B\|f\|^2.$$
 \end{itemize}
\end{proof}

\section{Results}\label{Sec-result}

In this section we discuss properties of atomic subspaces and characterize the same.

We recall the quotient of bounded operators (see \cite{Ra16}).
\begin{defn}
Let $A, B \in \mathcal L(\mathcal H)$ with $N(B) \subset N(A)$. The quotient operator $T= [A/B]$ is a map from $R(B)$ to $R(A)$ defined by $Bx \mapsto Ax$.
\end{defn}
It may be noted that $D(T)=R(B)$, $R(T) \subset R(A)$ and $TB=A$.

Usefulness of Bessel sequence in frame theory and in general in mathematical analysis is well known. Similarly the concept of fusion Bessel sequence gives us so many spin-off results in fusion frame theory. In the following two theorems (\ref{result-thm1-NS}, \ref{result-thm2-NS}) we present necessary and sufficient conditions for fusion Bessel sequence to be $K$-fusion frame.

\begin{thm}\label{result-thm1-NS} Let $\lbrace (\mathcal W_i ,v_i)\rbrace_{i \in I}$ be a fusion Bessel sequence in $\mathcal H$ with corresponding fusion frame operator $S_\mathcal W$ and assume that $K \in \mathcal L(\mathcal H)$. Then $\lbrace (\mathcal W_i ,v_i)\rbrace_{i \in I}$ is a $K$-fusion frame if and only if the quotient operator $\left[K^*/S_\mathcal W ^{1/2}\right]$ is bounded.
\end{thm}

\begin{proof} Let  $\lbrace (\mathcal W_i ,v_i)\rbrace_{i \in I}$ be a $K$-fusion frame. Then there is a constant $A_1 >0$ such that
\begin{equation}\label{eqn:result_Kfusion}
A_1\|K^*f\|^2 \leq \sum_{i \in I} v_i ^2 \|P_{\mathcal W_i} f\|^2 =\langle S_\mathcal W f, f \rangle = \|S_\mathcal W ^{1/2} f \|^2,
\end{equation}
for all $f \in \mathcal H$. Now let us denote the quotient operator $\left[K^* /S_\mathcal W ^{1/2}\right]$ by $T$. Then $T: R(S_\mathcal W ^{1/2}) \rightarrow R(K^*)$ such that $T(S_\mathcal W ^{1/2} f) = K^* f$ for all $f \in \mathcal H$. From \ref{eqn:result_Kfusion}, it is clear that $N(S_\mathcal W ^{1/2}) \subset N(K^*)$ and thus $T$ is well defined. Also $\|T(S_\mathcal W ^{1/2} f)\| = \|K^* f\| \leq \frac {1}{\sqrt{A_1}} \|S_\mathcal W ^{1/2} f \|$ for all $f \in \mathcal H$ and hence $T$ is bounded.

Conversely, suppose that the quotient operator $[K^* / S_\mathcal W ^{1/2}]$ is bounded. Then there exists a constant $A_2 >0$ such that $\|K^*f\|^2 \leq A_2 \|S_\mathcal W ^{1/2} f \|^2 = A_2\langle S_\mathcal W f, f\rangle  = A_2 \sum_{i \in I} v_i ^2 \|P_{\mathcal W_i} f \|^2$ for all $f \in \mathcal H$ and consequently  $\lbrace (\mathcal W_i ,v_i)\rbrace_{i \in I}$ forms a $K$-fusion frame for $\mathcal H$.
\end{proof}

\begin{cor} Let $\mathcal H$ be a Hilbert space. Let $\lbrace (\mathcal W_i ,v_i)\rbrace_{i \in I}$ be a fusion Bessel sequence in $\mathcal H$ with the fusion frame operator $\mathcal S_\mathcal W$. Then $\lbrace (\mathcal W_i ,v_i)\rbrace_{i \in I}$ is a fusion frame if and only if $\mathcal S_\mathcal W$ is invertible and positive.
\end{cor}

\begin{proof} One direction is obvious from the definition and the fact that $\langle S_{\mathcal{W}}f, f\rangle = \sum_{i \in I} v_i ^2 \|P_{\mathcal W_i} f\|^2$.

Conversely, let us assume that $\mathcal S_\mathcal W $ is invertible and positive.  Then the result follows from Theorem \ref{result-thm1-NS} with $K=I$.
\end{proof}

\begin{thm}\label{result-thm2-NS}
Let $\lbrace (\mathcal W_i ,v_i)\rbrace_{i \in I}$ be a fusion Bessel sequence in $\mathcal H$ with fusion frame operator $S_\mathcal W$ and $K\in \mathcal L(\mathcal H)$, then $\lbrace (\mathcal W_i ,v_i)\rbrace_{i \in I}$ is a $K$-fusion frame for $\mathcal H$ if and only if there exists a positive constant $A$ such that $S_\mathcal{W}\geq AK K^*$.
\end{thm}

\begin{proof}
The proof follows from the fact that $\langle S_\mathcal Wf, f\rangle = \sum_{i\in I}v_i^2\|P_{\mathcal{W}_i}f\|^2$. See \cite{Xi13} for details.
\end{proof}

Here we present a necessary and sufficient condition for a family of closed subspaces to be a $K$-fusion frame.

\begin{thm}\label{result-thm3-NS}
Let $\mathcal H$ be a Hilbert space and $K\in \mathcal L(\mathcal H)$. Assume that $\lbrace \mathcal W_i \rbrace_{i \in I}$ be a family of closed subspaces of $\mathcal H$ and $\lbrace  v_i \rbrace_{i \in I}$ be a family of positive weights. Then $\lbrace (\mathcal W_i ,v_i)\rbrace_{i \in I}$ is a $K$-fusion frame for $\mathcal H$ if and only if there exists a bounded, linear operator $L:(\sum_{i \in I} \oplus \mathcal W_i)_{l^2} \rightarrow \mathcal H$ such that $Le_n=\sum_{i \in I} v_i e_n ^i$ and $R(K) \subseteq R(L)$, where $ \lbrace e_n \rbrace_{n=1}^\infty$ is an orthonormal basis in $(\sum_{i \in I} \oplus \mathcal W_i)_{l^2}$ and $e_n ^i$ is the $i$-th component of $e_n$.
\end{thm}

\begin{proof} Let $\lbrace (\mathcal W_i ,v_i)\rbrace_{i \in I}$ be a $K$-fusion frame. Then there exist positive constants $A$ and $B$ such that  \begin{equation*}A\|K^*f\|^2 \leq \sum_{i \in I} v_i ^2 \|P_{\mathcal W_i} f\|^2 \leq B\|f\|^2,\end{equation*} for all $f \in \mathcal H$. Define $L^* : \mathcal H \rightarrow (\sum_{i \in I} \oplus \mathcal W_i)_{l^2}$ such that $L^*f = \lbrace v_i P_{\mathcal W_i} f \rbrace_{i \in I}$. Then $\|L^*f\|^2 = \sum_{i \in I} v_i ^2 \|P_{\mathcal W_i} f\|^2$ for all $f \in \mathcal H$. Hence by the previous inequality we have $A\|K^*f\|^2 \leq \|L^*f\|^2$ for all $f \in \mathcal H$ and therefore $AKK^* \leq LL^*$. Therefore by Theorem \ref{Thm-Douglas}, $R(K) \subset R(L)$. Now $ \langle f, Le_n \rangle_ \mathcal H=\langle L^*f,e_n \rangle_{(\sum_{i \in I} \oplus \mathcal W_i)_{l^2}}=\langle \lbrace v_i P_{\mathcal W_i} f \rbrace , e_n \rangle_{(\sum_{i \in I} \oplus \mathcal W_i)_{l^2}}=\sum_{i \in I} \langle v_i P_{\mathcal W_i} f , e_n ^i \rangle_\mathcal H=\sum_{i \in I} \langle f, v_i e_n ^i \rangle_\mathcal H$. Hence $Le_n = \sum_{i \in I} v_i e_n ^i$.

Conversely, suppose $L: (\sum_{i \in I} \oplus \mathcal W_i)_{l^2} \rightarrow \mathcal H$ such that $Le_n = \sum_{i \in I} v_i e_n ^i$ and $R(K) \subset R(L)$. Then $L^*f = \lbrace v_i P_{\mathcal W_i} f \rbrace_{i \in I}$. Therefore $\sum_{i \in I} v_i ^2 \|P_{\mathcal W_i} f\|^2=\|L^*f\|^2 \leq \|L^*\|^2 \|f\|^2$. Hence  $\lbrace (\mathcal W_i ,v_i)\rbrace_{i \in I}$ form a fusion Bessel sequence. Now since $R(K) \subseteq R(L)$, again by Theorem \ref{Thm-Douglas}, there exists a positive constant $A$ such that $AKK^* \leq LL^*$ and hence $A\|K^*f\|^2 \leq  \|L^*f\|^2 = \sum_{i \in I} v_i ^2 \|P_{\mathcal W_i} f\|^2$ . Consequently $\lbrace (\mathcal W_i ,v_i)\rbrace_{i \in I}$ is a $K$-fusion frame for $\mathcal H$.
\end{proof}

Following two results show methods of construction of $K$-fusion frames from $K$-frames. Analogous results for fusion frames are discussed in \cite{Ca04}.

\begin{thm}\label{resullt-thm4-kfr-kfufr}
 Let $\mathcal H$ be a Hilbert space, $K\in \mathcal L(\mathcal H)$ and $\lbrace f_j \rbrace_{j \in J}$ be a $K$-frame for $\mathcal H$ with frame bounds $A$ and $B$. Assume that $\lbrace J_i \rbrace_{i \in I}$  is a partition of the index set $J$ and $\mathcal W_i$ is the closed linear span of $\lbrace f_j \rbrace_{j \in J_i}$ for all $i \in I$. Then for all $f \in \mathcal H$ we have $\frac {A} {B} \|K^*f\|^2 \leq \sum_{i \in I} \|P_{\mathcal W_i} f\|^2$. Further if $|I| < \infty$ then $\lbrace \mathcal W_i \rbrace_{i \in I}$ is an $1$-uniform $K$-fusion frame for $\mathcal H$.
\end{thm}

\begin{proof} Since $\lbrace f_j \rbrace_{j \in J}$ is a $K$-frame for $\mathcal H$  with  bounds $A$ and $B$, we have $$A\|K^*f\|^2\leq \sum_{j\in J} |\langle f, f_j\rangle|^2 = \sum_{i\in I} \sum_{j\in J_i} |\langle f, f_j\rangle|^2 = \sum_{i \in I} \sum_{j \in J_i} | \langle P_{\mathcal W_i} f,f_j \rangle |^2\leq B\|f\|^2,$$ for all $f\in \mathcal H$. Now since every sub-collection of a Bessel sequence is also a Bessel, we have $\sum_{i \in I} \sum_{j \in J_i} | \langle P_{\mathcal W_i} f,f_j \rangle |^2 \leq \sum_{i \in I} B\|P_{\mathcal W_i} f\|^2 $. Hence we have $\frac {A} {B} \|K^*f\|^2 \leq \sum_{i \in I} \|P_{\mathcal W_i} f\|^2$, for all $f\in \mathcal H$.

Further, if $|I| < \infty$ then we have $\sum_{i \in I} \|P_{\mathcal W_i} f\|^2 \leq |I| \|f\|^2$. Hence in this special case $\lbrace \mathcal W_i \rbrace_{i \in I}$ is always an $1$-uniform $K$-fusion frame for $\mathcal H$.
\end{proof}

\begin{cor}\label{resullt-cor-kfr-kfufr}
 Let $\mathcal H$ be a Hilbert space and $K\in \mathcal L(\mathcal H)$. Suppose $\lbrace f_j \rbrace_{j \in J}$ is a $K$-frame for $\mathcal H$. Assume that $J=J_1 \cup J_2 \cup .... \cup J_n$ be a finite partition of $J$ and $\mathcal W_i = \overline{span_{j \in J_i} \lbrace f_j} \rbrace$. Then  $\lbrace (\mathcal W_i ,v_i)\rbrace_{i=1}^n$ forms a $K$-fusion frame for $\mathcal H$ for any collection of positive weights $\lbrace v_i \rbrace_{i=1}^n$.
\end{cor}

\begin{proof} Let $\lbrace f_j \rbrace_{j \in J}$ be a $K$-frame for $\mathcal H$ with frame bounds $A$ and $B$. Then using Theorem \ref{resullt-thm4-kfr-kfufr}, $\{W_i\}$ forms an $1$-uniform $K$-fusion frame with frame bounds $A/B$ and $n$. That is
$$\frac {A} {B} \|K^*f\|^2 \leq \sum_{i=1}^n \|P_{\mathcal W_i} f\|^2\leq n \|f\|^2,$$ for all $f\in \mathcal H$. Now considering $v=\min\{v_i:~i=1, 2, \cdots, n\}$ and $w=\max\{v_i:~i=1, 2, \cdots, n\}$, we have
$$\frac {Av^2} {B} \|K^*f\|^2 \leq \sum_{i=1}^n v_i^2\|P_{\mathcal W_i} f\|^2\leq nw^2 \|f\|^2,$$ for all $f\in \mathcal H$. Hence proved.

\end{proof}

\begin{defn}
Let $\lbrace \mathcal H_i \rbrace_{i \in I}$ be a non-overlapping family of Hilbert spaces. For each $i \in I$, let us assume that $T_i : \mathcal H_i \rightarrow \mathcal H_i$ be a bounded, linear operator on $\mathcal H_i$ such that the family $\lbrace T_i \rbrace_{i \in I}$ is uniformly bounded i.e. $sup \lbrace \|T_i\| : i \in I \rbrace < \infty$. Then the direct sum operator of the uniformly bounded family $\lbrace T_i \rbrace_{i \in I}$ is the operator $ \bigoplus_{i \in I} T_i : \bigoplus_{i \in I} \mathcal H_i \rightarrow \bigoplus_{i \in I} \mathcal H_i$ on the direct sum of the Hilbert spaces $\bigoplus_{i \in I} \mathcal H_i$ is defined as $(\bigoplus_{i \in I} T_i) (x)= \sum_{i \in I} T_i x_i$, where $x= \sum_{i \in I}  x_i$ and $x_i \in \mathcal H_i$.
\end{defn}
It is easy to check that $ \bigoplus_{i \in I} T_i$ is well defined, bounded, linear operator, whose norm is given by $\|\bigoplus_{i \in I} T_i \| = sup \lbrace \|T_i \| : i \in I \rbrace$.

In the following theorem we will show that direct sum of $K$-fusion frames is a $K$-fusion frame.

\begin{thm} Let $\{(\mathcal W_{ij}, v_i)\}_{i \in I}$ be a collection of $K_j$-fusion frames for $\mathcal H_j$,  $j=1, 2, \cdots, m$ with $\mathcal W_{ij} \cap \mathcal W_{ik} = \phi$ for $j \neq k$. Then $\left\{\bigoplus _{j=1}^m \mathcal W_{ij}, v_i\right\}_{i\in I}$ is a $\bigoplus_{j=1}^m K_j $-fusion frame for the Hilbert space $\bigoplus_{j=1}^m H_j$.
\end{thm}

\begin{proof} It is sufficient to prove the result for $m=2$. Let $A_j$ and $B_j$ be frame bounds for the $K_j$-fusion frame $\lbrace (\mathcal W_{ij}, v_i)\rbrace_{i \in I}$ $j=1, 2$. Since $P_{\mathcal W_{i1} \oplus \mathcal W_{i2} } = P_{\mathcal W_{i1}} \oplus P_{\mathcal W_{i2}}$ $(i \in I)$, then for all $f \in \mathcal H_1$ and $g \in \mathcal H_2$ we have,
\begin{eqnarray*}
\min \lbrace A_1, A_2 \rbrace \|(K_1 \oplus K_2)^* (f \oplus g)\|^2 &=& \min \lbrace A_1, A_2 \rbrace \|(K_1 ^* \oplus K_2 ^*) (f \oplus g)\|^2\\
& =& \min \lbrace A_1, A_2 \rbrace \|K_1 ^* f \oplus K_2 ^* g \|^2\\
& =& \min \lbrace A_1, A_2 \rbrace ( \|K_1 ^* f\|^2 + \|K_2 ^* g\|^2)\\
&\leq& A_1 \|K_1 ^* f\|^2 + A_2 \|K_2 ^* g\|^2 \\
&\leq& \sum_{i \in I} v_i ^2 \|P_{\mathcal W_{i1}} (f)\|^2 + \sum_{i \in I} v_i ^2 \|P_{\mathcal W_{i2}} (g)\|^2 \\
&\leq& B_1 \|f\|^2 + B_2 \|g\|^2 \\
&\leq& \max \lbrace B_1, B_2 \rbrace (\|f\|^2 + \|g\|^2) \\
&=&  \max \lbrace B_1, B_2 \rbrace \|f \oplus g\|^2.
\end{eqnarray*}
Result follows from the fact that $\sum_{i \in I} v_i ^2 \|P_{\mathcal W_{i1}} (f)\|^2 + \sum_{i \in I} v_i ^2 \|P_{\mathcal W_{i2}} (g)\|^2 = \sum_{i \in I} v_i ^2 \|P_{\mathcal W_{i1} \oplus \mathcal W_{i2} } (f \oplus g)\|^2$.
\end{proof}

In the following result we will present some algebraic properties of $K$-fusion frame.

\begin{thm} Let $K_j\in \mathcal L(\mathcal H)$ and $\{a_j\}$ be a finite collection of scalars for $j=1, 2, \cdots, n$. Suppose $\lbrace (\mathcal W_i, v_i)\rbrace_{i \in I}$ is a $K_j$-fusion frame for $\mathcal H$, for all $j=1, 2, \cdots, n$. Then $\lbrace (\mathcal W_i, v_i)\rbrace_{i \in I}$ is also a $\sum_{j=1}^n a_j K_j$-fusion frame and $\prod_{j=1}^n K_j$-fusion frame for $\mathcal H$.
\end{thm}

\begin {proof} Since $\lbrace (\mathcal W_i, v_i)\rbrace_{i \in I}$ is a $K_j$-fusion frame for $\mathcal H$, for all $j$, there exist $A, B>0$ such that $$A\|K_j^* f \|^2 \leq \sum_{i\in I}v_i^2\|P_{\mathcal W_i}f\|^2 \leq B \|f\|^2.$$
Then the conclusion follows from the following inequalities:
$$\dfrac{A}{(\sum\limits_j|a_j|)^2}\|(\sum_{j=1}^na_j K_j)^* f \|^2 \leq \sum_{i\in I}v_i^2\|P_{\mathcal W_i}f\|^2 \leq B \|f\|^2,$$
and
$$\frac {A} {\prod_{j=2}^n \|K_j ^* \|^2} \|(\prod_{j=1}^n K_j )^* f\|^2 \leq A \|K_1 ^* f \|^2 \leq \sum_{i\in I}v_i^2\|P_{\mathcal W_i}f\|^2 \leq B\|f\|^2 ,$$
for all $f\in \mathcal H$. It may be noted that the trivial case, $K_j$ being zero operator, has been omitted.
\end{proof}

Suppose $\mathcal U$ and $\mathcal V$ are two closed subspaces of $\mathcal H$ and $P_\mathcal U,~ P_\mathcal V$ are orthogonal projections from $\mathcal H$ onto $U,~V$, respectively, such that $P_\mathcal UP_\mathcal V = P_\mathcal VP_\mathcal U$. Then it is well-known that $P_\mathcal U P_\mathcal V$ is the orthogonal projection from $\mathcal H$ onto $\mathcal U\cap \mathcal V$. In the following we will discuss when the intersection of $K$-fusion frames is a $K$-fusion frame.

\begin{lem}\label{result-lem-intersection}
Suppose that $\lbrace \mathcal W_i \rbrace_{i \in I}$, $\lbrace \mathcal V_i \rbrace_{i \in I}$ are families of closed subspaces of $\mathcal H$ and $\lbrace  w_i \rbrace_{i \in I}$, $\lbrace  v_i \rbrace_{i \in I}$ are families of positive weights. Also suppose that the orthogonal projections $P_{\mathcal W_i}~ \&~ P_{\mathcal V_i}$ commute for each $i\in I$.  If $\lbrace (\mathcal W_i ,w_i)\rbrace_{i \in I}$ (or $\lbrace (\mathcal V_i, v_i)\rbrace_{i \in I}$) is a fusion Bessel sequences in $\mathcal H$, then so is $\lbrace (\mathcal W_i \cap \mathcal V_i,w_i)\rbrace_{i \in I}$ (or $\lbrace (\mathcal W_i \cap \mathcal V_i,v_i)\rbrace_{i \in I}$).
\end{lem}

\begin{proof}
Suppose $\lbrace (\mathcal W_i, w_i)\rbrace_{i \in I}$ is a fusion Bessel sequence. Then for some constant $B >0$, we have for all $f\in \mathcal H$
$$\sum_{i \in I} w_i ^2 \|P_{\mathcal W_i \cap \mathcal V_i} f\|^2 = \sum_{i \in I} w_i ^2 \|P_{\mathcal V_i } P_{\mathcal W_i} f\|^2 ~  ~ \leq  \sum_{i \in I} w_i ^2 \| P_{\mathcal W_i} f\|^2 ~ \leq B \|f\|^2$$ and hence $\lbrace (\mathcal W_i \cap \mathcal V_i, w_i)\rbrace_{i \in I}$ is a fusion Bessel sequence.
\end{proof}

\begin{thm}
Let $\lbrace (\mathcal W_i, w_i)\rbrace_{i \in I}$ be a fusion frame for $\mathcal H$ and $\mathcal V$ be a closed subspace of $\mathcal H$. Also assume $P_{\mathcal V}$ commutes with $P_{\mathcal W_i}$ for each $i\in I$. Then $\lbrace (\mathcal W_i \cap \mathcal V,w_i)\rbrace_{i \in I}$ will form a $P_{\mathcal V}$-fusion frame for $\mathcal H$.
\end{thm}

\begin{proof} Suppose $\lbrace (\mathcal W_i, w_i)\rbrace_{i \in I}$ is a fusion frame for $\mathcal H$, then for some constants $A, B>0$ and using Lemma \ref{result-lem-intersection} we have $$A\|P_{\mathcal V} ^* f\|^2 = A\|P_{\mathcal V} f\|^2 ~ \leq \sum_{i \in I} w_i ^2 \|P_{\mathcal W_i } P_{\mathcal V} f\|^2 = \sum_{i \in I} w_i ^2 \|P_{\mathcal W_i \cap \mathcal V} f\|^2 ~\leq B\|f\|^2,$$ for all $f \in \mathcal H$. Hence  $\lbrace (\mathcal W_i \cap \mathcal V,w_i)\rbrace_{i \in I}$ is a  $P_{\mathcal V}$-fusion frame for $\mathcal H$.
\end{proof}

\begin{thm}  Let $\lbrace (\mathcal W_i, w_i)\rbrace_{i \in I}$ be a $K$-fusion frame for $\mathcal H$ where  $K \in \mathcal L(\mathcal H)$ and $\mathcal V$ be a closed subspace of $\mathcal H$. Also assume that $P_{\mathcal V}$ commutes with $P_{\mathcal W_i}$ for each $i\in I$ and $P_{\mathcal V}^{\dag}$ commutes with $K^*$. Then $\lbrace (\mathcal W_i \cap \mathcal V,w_i)\rbrace_{i \in I}$ forms a $K$-fusion frame for $R(P_{\mathcal V})$.
\end{thm}

\begin{proof} Since $P_{\mathcal V}$ has closed range, $P_{\mathcal V}^\dag$ exists. $\lbrace (\mathcal W_i, w_i)\rbrace_{i \in I}$ is a $K$-fusion frame for $\mathcal H$ implies that there exist positive constants $A, B$ such that
$$A\|K^*f\|^2\leq\sum_{i\in I}w_i^2\|P_{\mathcal W_i}f\|^2\leq B\|f\|^2,$$for all $f\in \mathcal H$. Therefore using Lemma \ref{result-lem-intersection}, for all $f\in R(P_{\mathcal V})$, we have
$\frac{A}{\|P_\mathcal V ^ \dag\|^2} \|K^*f\|^2=\frac{A}{\|P_\mathcal V ^ \dag\|^2} \|K^*P_\mathcal V ^\dag P_\mathcal V f\|^2\leq A\|K^*P_\mathcal V f\|^2\leq \sum_{i\in I}w_i^2\|P_{\mathcal W_i}P_\mathcal V f\|^2 = \sum_{i\in I}w_i^2\|P_{\mathcal W_i\cap \mathcal V} f\|^2\leq B \|f\|^2.$ Hence $\lbrace (\mathcal W_i \cap \mathcal V, w_i)\rbrace_{i \in I}$ is a $K$-fusion frame for $R(P_{\mathcal V})$.
\end{proof}

\section{Conclusion}
In the area of frame theory, the study of atomic subspaces has a great significance to characterize fusion frames with respect to a bounded linear operator, which we have analyzed in Sections \ref{Sec-atom} \& \ref{Sec-result}.

$K$-fusion frames come naturally when one needs to reconstruct functions from a large data in the range of a bounded linear operator. $K$-fusion frames can be further studied to rich the existing literature of fusion frames and their applications in coding theory, sensor network, etc.

\section*{Acknowledgements}
The first author acknowledges the financial support of MHRD, Government of India.


\begin{thebibliography}{99}

%
%

\bibitem{Ca04} P. Casazza and G. Kutyniok, Frames of subspaces, \textit{ Contemp. Math., AMS} \textbf{345} (2004), 87--114.


\bibitem{Ca08} P. Casazza, G. Kutyniok and S. Li, Fusion frames and distributed processing, \textit{Appl. Comput. Harmon. Anal.}, \textbf{25} (2008), 114--132.
    
\bibitem{Ch03} O. Christensen, \textit{An introduction to frames and riesz bases}, Birkh{\"a}user, Boston, 2003.

\bibitem{Da86} I. Daubechies, A. Grossmann  and Y. Meyer, Painless nonorthogonal expansions, \textit{J. Math. Phys.}, \textbf{27}(5) (1986), 1271--1283.

\bibitem{Do66} R. G. Douglas, On majorization, factorization and range inclusion of operators on {H}ilbert space, \textit{Proc. Amer. Math. Soc.}, \textbf{17}(2) (1966), 413--415.

\bibitem{Du52} R. Duffin and A. C. Schaeffer, A class of nonharmonic Fourier series, \textit{Trans. Amer. Math. Soc.} \textbf{72}(2) (1952), 341--366.

\bibitem{El03} Y. C. Eldar, Sampling with arbitrary sampling and reconstruction spaces and oblique dual frame vectors, \textit{ J. Fourier Anal. Appl.}, \textbf{9}(1) (2003), 77--96.

\bibitem{Fe99} P. Ferreira, Mathematics for multimedia signal processing II: discrete finite frames and signal reconstruction, \textit{Signal Processing for Multimedia}, J. S. Byrnes, Ed. Amsterdam, The Netherlands: IOS Press (1999), 35--54.

\bibitem{Ga12} L. G\v{a}vru\c{t}a, Frames for operators, \textit{Appl. Comput. Harmon. Anal.}, \textbf{32}(1) (2012), 139--144.

\bibitem{Ga15} L. G\v{a}vru\c{t}a, Atomic decompositions for operators in reproducing kernel {H}ilbert spaces, \textit{Math. Rep.}, \textbf{17(67)}(3) (2015), 303--314.

\bibitem{Ho04} R. B. Holmes and V. I. Paulsen, Optimal frames for erasures, \textit{Linear Algebra Appl.}, \textbf{377} (2004), 31--51.

\bibitem{Jo13} S. Jose and K. C. Sivakumar, Chapter 10: Moore-Penrose inverse of perturbed operators on Hilbert spaces, In: \textit{Combinatorial matrix theory and generalized inverses of matrices}, (R. B. Bapat, S. J. Kirkland and K. M. Prasad, eds.), 119-131, Springer, NewYork, 2013.

\bibitem{Li18} A. F. Liu and P. T. Li, K-fusion frames and the corresponding generators for unitary systems, \textit{Acta Math. Sin. (Engl. Ser.)}, \textbf{34}(5) (2018), 843--854.

\bibitem{Ra16} G. Ramu and P. Johnson, Frame operators of $K$-frames, \textit{ SeMA J. }, \textbf{73}(2) (2016), 171--181.

\bibitem{St03} T. Strohmer and R. W. Heath Jr., Grassmannian frames with applications to coding and communication, \textit{Appl. Comput. Harmon. Anal.}, \textbf{14}(3) (2003), 257--275.

\bibitem{Xi13} X. Xiao, Y. Zhu and L. G\v{a}vru\c{t}a, Some properties of $K$-frames in Hilbert spaces, \textit{Results Math.}, \textbf{63}(3) (2013), 1243--1255.

\end{thebibliography}


\end{document}